\documentclass{amsart}
\usepackage{amsmath,amssymb,amsthm}
\usepackage{booktabs} 
\usepackage{graphicx}
\usepackage[shortlabels]{enumitem}
\usepackage{listings}
\usepackage{xcolor}
\usepackage{tikz}
\usepackage{hyperref}

\theoremstyle{plain}
\newtheorem{theorem}{Theorem}

\newtheorem{lemma}[theorem]{Lemma}

\theoremstyle{definition}

\newtheorem{remark}[theorem]{Remark}
\newtheorem*{remark*}{Remark}

\newcommand{\pr}{\mathbb{P}}

\newcommand{\E}{\mathbb{E}}

\newcommand{\Z}{\mathbb{Z}}

\title{A model of discrete interacting updates}

\author[Denisov]{Denis Denisov}
\address{Department of Mathematics, University of Manchester, UK}
\email{denis.denisov@manchester.ac.uk}

\author[Shneer]{Seva Shneer}
\address{School of MACS, Heriot-Watt University}
\email{v.shneer@hw.ac.uk}

\author[Wachtel]{Vitali Wachtel}
\address{Faculty of Mathematics, Bielefeld University, Germany}
\email{wachtel@math.uni-bielefeld.de}



\thanks{Funded by the Deutsche Forschungsgemeinschaft (DFG, German Research
Foundation) –
Project-ID 317210226 – SFB 1283.}

\begin{document}

\begin{abstract}
We consider $N$ counters taking integer values which are subject to the following dynamics. At every time, a pair of distinct counters is chosen uniformly at random and their states are updated according to the following rule. If the states are different, then the smaller one is increased by $1$, while if the states are the same, both of them are increased by $1$. We show that, for a fixed $N$, the distances between consecutive ordered counters form a positive recurrent Markov chain and there exists the speed $V(N)$ defined as the average number of counters updated per time step in the stationary regime. We provide non-trivial upper and lower bounds for $V(N)$ as $N\to \infty$. Despite the simple formulation of the problem, its analysis seems to be highly complicated. We also provide a list of open problems and discuss various methods one may want to use, and obstacles one encounters. 
\end{abstract}

\maketitle

\section{Introduction}

We consider the following model. There are $N$ counters that take integer values. At every time step a pair of distinct counters is selected uniformly at random, independently over steps. If states of the selected counters are different, the lower one is increased by $1$. If the states are equal, both of them are increased by $1$. Therefore, the total increase per step is random and takes values $1$ or $2$. We are interested in the average increase per step. We show that for every fixed $N$ this is well defined, and denote it by $V(N)$. We study the behaviour of this quantity as $N \to \infty$.

A simple observation is that, if the system starts from any configuration with at least two counters at the top level, then at any time step there are at least $2$ counters at the top level. Thus, it is obvious that $V(2)=2$. This observation also implies that the evolution of the system with $N=3$ may be described by a one-dimensional Markov chain tracking the distance of the lowest counter from the top level. Let us denote this distance after time step $t$ by $X(t)$. It is easy to see that $X(t)$ is a reflected random walk with transition probabilities given by
$$
\pr(X(t+1)=1|X(t)=0)=1,
$$
$$
\pr(X(t+1)=k+1|X(t)=k)=\frac13 
$$
and
$$
\pr(X(t+1)=k-1|X(t)=k)=\frac23
$$
for all $k\ge 1$.

In particular, one can see that this Markov chain is positive recurrent. Its stationary distribution is given by $\pi(0)=1/4, \pi(k) = 3/2^{k+2}, k\ge 1$. This implies that 
$$
V(3)= 2 \pi(0) + \frac43 (1-\pi(0)) = \frac32,
$$
as the speed in the configuration with all counters at the top level is equal to $2$, while in any other configuration it is equal to $4/3$.

In the case $N=4$ the evolution of the system may be described by a two-dimensional random walk in the positive quadrant reflected at its boundaries which tracks the distances from the top two counters to the third largest one and between the third largest one and the fourth (lowest) one. In this case one can also write down transition probabilities but the stationary distribution does not seem to have a closed-form expression. We will treat this case in Subsection~\ref{sebsec:N=4} where we show, in particular, that 
$$ 
\frac{26}{19} \le V(4) \le \frac{10}{7}.
$$
The same interpretation is possible for any $N$. More precisely, one can track distances between consecutive counters, starting from the second-highest one (which is at the same level as the highest one). As a result, we will obtain a reflecting random walk in the positive orthant of the $(N-2)$-dimensional space. The complexity of transition probabilities however grows rapidly with dimension and we do not expect there to be a way to treat these analytically. In particular, transition probabilities are different in each part of the reflecting boundary.

We show (see Theorem~\ref{thm:positive.recurrent}) that for every $N$, the Markov chain described above is positive recurrent and thus $V(N)$ is well defined for every $N$. We are interested in the behaviour of $V(N)$ as $N \to \infty$. Numerical experiments strongly suggest that $V(N)$ is a decreasing function of $N$. Thus $V(N)$ has a limit $V(\infty)$, and numerical experiments suggest that 
$$
V(\infty) \approx 1.242.
$$

It is still unclear to us how to prove these properties of $V(N)$, and these present exciting open problems. In this paper, we show in Theorem~\ref{thm.lower.bound} that
$$
\liminf_{N\to \infty} V(N) \ge 1 + \frac{\sqrt{3}}{27} \approx 1.064
$$ 
and we show in Theorem~\ref{thm.upper.bound} that
$$
\limsup_{N\to \infty} V(N) \le \frac{34}{27} \approx 1.259.
$$
The lower bound in particular implies that there is a clustering of counters in some levels, that is, the reflecting random walk spends an asymptotically non-vanishing proportion of time at boundaries.

Another possible approach is to consider tails of empirical measures (or proportions of counters above every fixed level)
$$
\psi^{(N)}_k(t) = \frac{1}{N}\sum_{i=1}^N I\left(C_i\left(\lfloor tN\rfloor\right)\ge k\right),
$$
where $C_i(s)$ is the level of counter $i$ at time $s$ and $\lfloor s \rfloor$ denotes the integer part of $s$. Using methods that have become standard (see~\cite{VvedenskayaDobrushinKarpelevich1996}), one can show that $\{\psi^{(N)}_k\}_{k\ge 0}$ converge, as $N \to \infty$, to the solutions of an infinite system of ODEs
$$
\psi_0(t)=1 \text{ for all } t\ge 0,
$$
\begin{equation} \label{eq:main_diff_eq}
\psi'_{k+1}(t) = 2\psi_k(t)(\psi_k(t)-\psi_{k+1}(t)) \text{ for } t\ge 0, k\ge 0
\end{equation}
with initial conditions
$$
\psi_k(0)=0 \text{ for } k\ge 1.
$$
Despite the availability of a recursive solution, it is not clear how one can exploit it to gain insight into the behaviour of the system as $k \to \infty$. We believe that, and numerical experiments strongly support this hypothesis, there exist a constant $c$ and a decreasing function $H$ such that
$$
\psi_k\left(\frac{k+x}{c}\right) \to H(x)
$$
as $k \to \infty$. This function $H$ should be a travelling-wave solution to the system of equations above such that it is decreasing and satisfies
$$
cH'(x) = H(x+1)(H(x+1)-H(x)), \quad H(-\infty)=0, H(\infty)=1.
$$
Usually in dealing with travelling-wave solutions, the value of $c$ is either known apriori or may be derived from the equation satisfied by the travelling wave. One is then interested in the properties of $H$.

The authors of \cite{PolterovichHenkin1988} consider an infinite system of differential equations that leads to the equation
\begin{equation} \label{eq:TW_HP}
cH'(x) = (a+bH(x))(H(x+1)-H(x)),
\end{equation}
where $c$ is the initially unknown speed and $a$ and $b$ are parameters of the model. Dividing both sides by $a+bH(x)$ and integrating, then taking the limit as $x\to \infty$, they are able to derive the exact value of $c$. They also derive the closed-form expression of $H$ satisfying the differential equation. In  further papers \cite{HenkinPolterovich1991,HenkinPolterovich1994}, they find the speed and show existence of a travelling-wave solution in a more general model. It follows from results of \cite{PolterovichHenkin1988} and \cite{HenkinPolterovich1991} that $H$ has exponential tails. This is a behaviour also observed in travelling-wave solutions to the Burgers equation. One can also see similarities between the Burgers equation and \eqref{eq:TW_HP}.

Papers \cite{PolterovichHenkin1988} and \cite{HenkinPolterovich1991} do not consider any stochastic models and deal directly with infinite systems of differential equations. 

A stochastic model where one can prove the existence of a travelling-wave solution and find the exact value of its speed appears from a simplification of our model. If we assume that, upon finding two counters with the same level, we increase only one of them, say chosen uniformly at random. This model describes the growth of queues in the so-called power-of-$2$ queueing system (see \cite{VvedenskayaDobrushinKarpelevich1996}) without service. It can be seen that the speed $c$ and the travelling wave $H$ satisfy
\begin{equation} \label{eq:TW_power_of_2}
cH'(x) = H^2(x+1)-H^2(x).
\end{equation}
By integrating the equation above and using boundary conditions, one can derive that $c=1$, which agrees with the intuition. The same equation has been studied in \cite{BenNaimKrapivsky2012} who argued that one of the tails of $H$ should be doubly exponential and the other exponential. A double exponential tail has been observed in travelling-wave solutions to the F-KPP equation, see~\cite{McKean1975}.  

In our case, although the differential equations seem to be very similar to the ones leading to exact expressions for $c$, it does not appear to be possible to find the value of the speed. We thus also lack techniques to study the behaviour of $H$. Numerical experiments suggest that $c=V(\infty)$.

The problem has originated in engineering applications, see, e.g., \cite{bianchi2012modeling} and references therein. The model is also equivalent to the pure growth process in the so-called redundancy system with cancel-on-completion mechanism when all service times are deterministic and equal to $1$. Such systems, with general service times, have recently been studied in \cite{stolyar2026large}. The existence of a limiting speed and a travelling-wave solution in our model follows from very general results of \cite{stolyar2026large}. These results do not however seem to provide further insight into the properties of the speed or the travelling-wave solution. One cannot for instance conclude whether the limiting speed is strictly larger than $1$, or what are the tails of the travelling-wave solution. The goal of our research has been to utilise the structure of the process to obtain more constructive results. Some questions still remain unanswered, and we will discuss these in Section \ref{subsec:open}.

\section{Main results and discussions}

In this section, we first formulate our main results. We then discuss the case $N=4$ that illustrates the technical difficulties one faces when dealing with this model. We then describe our approach in the general case. We finish this section with a list of open problems related to this model, including the mean-field approach that may initially look promising for dealing with the case $N\to \infty$. We describe the obstacles in this approach we were not able to overcome.

\subsection{Main results}

We will assume that the system starts from a configuration such that there are at least $2$ counters at the highest level (this does not restrict the generality as such a state will be reached in a bounded time with probability $1$). In this case, as is easy to see, there are always at least two counters at the highest level.
Let $X_j$ be the distance between counters $j+2$ and $j+1$, assuming they are ordered from largest to smallest. 
Markov chain $X=(X(t)=(X_1(t),\ldots, X_{N-2}(t)))_{t\ge 0}$ fully describes the evolution of the system. The theorem below shows positive recurrence of $X$ which, in particular, allows us to express the time-average increase $V(N)$ as the average increase with respect to the stationary distribution.
\begin{theorem}\label{thm:positive.recurrent}
For each $N\ge 3$, $X$ is a positive recurrent  Markov chain  with stationary distribution $\pi$.
\end{theorem}

It is worth to mention that the recurrence/transience classification for general random walks in orthants is a complicated task, see \cite[Part II]{Borovkov1998}.

Our main results concern upper and lower bounds for $V(N)$ when $N$ is large.

\begin{theorem}\label{thm.upper.bound}
As $N\to \infty$, 
\[
V(N)\le (1+o(1))
\left(
1+\frac{7}{27}
\right).
\]
\end{theorem}

\begin{theorem} \label{thm.lower.bound}
As $N\to \infty$, 
\[
V(N)\ge (1+o(1))
\left(1+\frac{\sqrt{3}}{27}
\right).
\]
\end{theorem}

\begin{remark}\label{rem.finite.case}
We also have the following non-asymptotic upper bounds for 
$N\ge 5$:
\[
V(N)\le 1+
\begin{cases}
\dfrac{N(7N-9)}{27\,(N^{2}-3N+2)}, & N\equiv0\pmod3,\\[8pt]
\dfrac{7N^{2}-2N-5}{27\,N(N-2)}, & N\equiv1\pmod3,\\[8pt]
\dfrac{7N^{3}-9N^{2}-3N+13}{27\,N(N^{2}-3N+2)}, & N\equiv2\pmod3.
\end{cases}
\]
We prove these in the appendix. We believe that one can also construct non-asymptotic lower bounds by following a similar strategy of analysing worst-case scenarios.
\end{remark}


\subsection{Case $N=4$}\label{sebsec:N=4}
Before turning to proofs, we will discuss in detail the case $N=4$ which leads to a two-dimensional reflected random walk and demonstrates the difficulties one encounters in the analysis of this model even in a relatively low dimension. Random walks in the positive quadrant is a popular topic in probability and combinatorics, we refer the reader to \cite{MR3617630} for an extensive exposition of results in the area.

The transition rules for the Markov chain $(X_1(t),X_2(t))$ are different depending on the current position. Let $p((i,j),(k,l))$ denote  the transition probability from $(i,j)$ to $(k,l)$. Then, on the boundaries we have
$$
p((0,0),(1,0))=1,
$$
$$
p((k,0),(k+1,0)) = p((k,0),(k-1,0))=\frac16,\quad p((k,0),(k-1,1))=\frac23, \quad k\ge 1,
$$
$$
p((0,l),(0,l-1)) = p((0,l),(1,l))=\frac12, \quad k\ge 1,
$$
while in the interior
$$
p((k,l),(k+1,l))=\frac16, \quad p((k,l),(k,l-1))=\frac12, \quad p((k,l),(k-1,l+1))=\frac13
$$
for all $k,l\ge 1$. The figure below provides an illustration of these transition probabilities.
\vspace{5mm}

\begin{tikzpicture}[>=stealth, thick, scale=1.5]

\draw[->] (0,0) -- (0,3) ;
\draw[->] (0,0) -- (6,0) ;

\fill (0,2) circle (2pt);
\draw[->] (0,2) -- (0.8,2) node[midway,above] {$\tfrac12$};
\draw[->] (0,2) -- (0,1.2) node[midway,left] {$\tfrac12$};

\fill (4,0) circle (2pt);
\draw[->] (4,0) -- (4.8,0) node[midway,above] {$\tfrac16$};
\draw[->] (4,0) -- (3.2,0) node[midway,above] {$\tfrac16$};
\draw[->] (4,0) -- (3.4,0.6) node[midway,above] {$\tfrac23$};

\fill (3,2) circle (2pt);

\draw[->] (3,2) -- (3.8,2) node[midway,above] {$\tfrac16$};
\draw[->] (3,2) -- (3,1.2) node[midway,right] {$\tfrac12$};
\draw[->] (3,2) -- (2.4,2.6) node[midway,left] {$\tfrac13$};

\fill (0,0) circle (2pt);
\draw[->] (0,0) -- (0.8,0) node[midway,above] {$1$};

\end{tikzpicture}
\vspace{5mm}

Let us denote stationary probabilities of being in different parts of the space as
\begin{align*}
& \Pi_0 = \pi((0,0)),\\
& \Pi_1 = \pi(\{(k,0),k\ge 1\}),\\
& \Pi_2 = \pi(\{(0,l),l\ge 1\}),\\
& \Pi_3 = \pi(\{(k,l),k,l\ge 1\}).
\end{align*}
Equating the frequency of jumps left and right in the stationary regime, we obtain
$$
\Pi_0 + \frac16 \Pi_1 + \frac12 \Pi_2 + \frac16 \Pi_3 = \frac16 \Pi_1 + \frac23 \Pi_1 + \frac13 \Pi_3,
$$
which simplifies to
\begin{equation} \label{eq:4_stationary_1}
\Pi_0 - \frac23 \Pi_1 + \frac12 \Pi_2 - \frac16 \Pi_3 = 0.
\end{equation}
Similarly, equating frequencies of jumps up and down, we obtain
$$
\frac23 \Pi_1 + \frac13 \Pi_3 = \frac12 \Pi_2 + \frac12 \Pi_3,
$$
which gives
\begin{equation} \label{eq:4_stationary_2}
\frac23 \Pi_1 - \frac12 \Pi_2 - \frac16 \Pi_3 = 0.
\end{equation}
Note that the average number of updated counters in any of the $4$ regions is constant. Then the speed may be written as
$$
V(4) = 1 + \Pi_0 + \frac13 \Pi_1 + \frac12 \Pi_2 + \frac16 \Pi_3.
$$
Together with the requirement that the sum of all probabilities is equal to $1$, equations \eqref{eq:4_stationary_1} and \eqref{eq:4_stationary_2} provide a way to express the speed in terms of one of the $\Pi$'s. For instance,
$$
V(4) = \frac{10}{7} - \frac27 \Pi_0,
$$
leading to the estimate
$$
V(4) \le \frac{10}{7} \approx 1.4286,
$$
which turns out to be the best possible upper bound arising from this system. The expression for $V(4)$ above also immediately implies $V(4) \ge 8/7$.
It turns out that a more careful treatment of the system allows one to improve this to
$$
V(4) \ge \frac{26}{19} \approx 1.3684.
$$
Unfortunately, it does not seem to be possible to calculate $V(4)$ exactly. In particular, there does not seem to be a closed-form expression for the stationary distribution of this two-dimensional Markov chain. Numerical simulations suggest that $V(4) \approx 1.3958$.

It is clear that, for an arbitrary $N$, the number of regions with different transition probabilities is $2^{N-2}$, while the number of equations connecting them is $N-1$. It would thus appear that looking for estimates similar to those explained above should be a hopeless task. However, we will demonstrate that non-trivial bounds for the speed may nevertheless be constructed via appropriate test functions.

\subsection{Description of our approach}
In order to classify the various regions, we will use the following simplified formulation of the state space which includes all the information needed for the calculation of the speed.

Let $x=(x_1,\ldots,x_{N-2})$ be 
a position of Markov chain $X$. 
We assign a configuration 
$\alpha=(\alpha_1,\ldots,\alpha_M)$ to $x$ as follows.  

Suppose there are $M\ge 1$ distinct levels of counters, with  the following numbers of counters at different levels: 
$$
\begin{array}{ll}
00 \dots 0 & \alpha_1 \text{ counters} \\
00 \dots 000 & \alpha_2 \text{ counters} \\
000 \dots 00 & \alpha_3 \text{ counters} \\
\text{\ldots} & \text{ and so on}
\end{array}
$$
Here $\alpha_1 \ge 2$ and $\alpha_2,\alpha_3,\ldots, \alpha_M \ge 1$. 
Note that 
level $1$ is treated 
differently and $\alpha_1\ge 2$ because we always 
have at least two counters in level $1$. 

This formulation is indeed a simplification of the state space as we ignore the values of distances between different levels.

If two points $x$ and $y$ have the same configuration $\alpha$, then the increments of the Markov chain in these two points have the same distributions. 
Thus, for each configuration $\alpha$ there is a subset $C_\alpha$
of $\Z_+^{N-2}$, where $X$ behaves like a random walk.  
Let $\pi_\alpha = \pi(C_\alpha)$. Then, it is clear that 
the average number of counters that moved one step up 
    is given by 
\[
V(N)= \sum_{\alpha} \pi_\alpha V_\alpha(N),
\]
where 
\[
V_\alpha(N) = 1+\sum_{i=1}^M \frac{\alpha_i(\alpha_i-1)}{N(N-1)}
\]
is the average number of counters updated if $X$ is in configuration $\alpha$. Note that $1$ on the RHS above corresponds to the fact that there is always at least one update. The second term corresponds to the probability of choosing two counters from the same level, which is the only possibility to update $2$ counters.

We aim to find non-trivial upper and lower bounds for $V(N)$. 
For that we will construct test functions $f$ and $g$ and determine constants $v_f$ and $v_g$ satisfying
\[
v_f+\sum_{k=1}^{N-2} f(k) \E[X_k(1)-x|X(0)=x] \ge V_\alpha(N)
\]
and
\[
v_g+\sum_{k=1}^{N-2} g(k) \E[X_k(1)-x|X(0)=x] \le V_\alpha(N),
\]
for all $x$ in configuration $\alpha$. These inequalities imply that 
\[
V(N)\le v_f \quad \text{and} \quad V(N)\ge v_g.
\]
Indeed, 
\begin{align*}
V(N) &= \sum_{\alpha} \pi_\alpha V_\alpha(N) \le 
\sum_{\alpha} \pi_\alpha\left(v_f+\sum_{k=1}^{N-2} f(k) \E[X_k(1)-x_k|X_0=x]\right)\\
&= v_f + \sum_{k=1}^{N-2} f(k)\sum_{\alpha} \pi_\alpha \E[X_k(1)-x_k|X_0=x]\\
&= v_f +\E_\pi \sum_{k=1}^{N-2} f(k) 
\left(X_k(1)-X_k(0)\right)=v_f. 
\end{align*}
The lower bound follows similarly.

We will take  
\begin{equation}\label{defn.f}
f(x)=\frac{x+1}{2}-\frac{x^2-1}{2(N-2)}
\end{equation}
and
\begin{equation}\label{defn.g}
g(x)=\left(\frac13 +\frac{\sqrt{3}}{18}\right)(x+1)-\frac{x^2-1}{2(N-2)}.
\end{equation}

To implement this plan, we will compare different configurations and perform worst-case-scenario analysis. This does not require any complicated techniques.

\subsection{Open problems} \label{subsec:open}

We finish this section by providing a list of open problems that we find exciting. We note again that the existence of a limiting speed and a travelling-wave solution follows from results of \cite{stolyar2026large}. We are interested in obtaining a more constructive understanding of this particular model and deriving explicit limits, bounds and tails of the quantities of interest. From this standpoint, most questions remain open.

\noindent {\bf 1. Monotonicity of speeds in $N$.} We have already alluded to the fact that numerical evidence strongly suggest that $V(N)$ is a decreasing function of $N$. This is an interesting open problem. If monotonicity was established, it would imply that $V(N)$ has a limit as $N\to \infty$ which, combined with our bounds, would mean that this limit is strictly larger than $1$. This would mean that the system spends a non-negligible amount of time on the boundaries - something not at all obvious intuitively. We were only able to show that $V(4)<V(3)$, and we did so analytically, using stationary distributions. It does not seem that there is a hope in generalising it. A more constructive, amenable to higher dimension proof, based on, e.g., trajectory-wise coupling, remains elusive even in this seemingly simple case.

\noindent {\bf 2. Exact expression for the limiting speed.} We do not have any heuristic argument that would lead to an expression for the speed. Numerically, it seems that $V(\infty) \approx 1.2422$.

Along with speed, there are many other unanswered questions related to the limiting system, including the asymptotic behaviour of the number of significant levels (those with many counters), and the distribution of counters over these levels.

\noindent {\bf 3. Mean-field approach.} We now discuss in more detail the infinite system of differential equations \eqref{eq:main_diff_eq}. A rigorous proof for the convergence of tails of empirical measures of the process to its solutions may be provided using techniques that became standard by now (we refer the reader to \cite{VvedenskayaDobrushinKarpelevich1996} and many follow-up papers). Heuristically it is straightforward to understand the differential equations as $\psi_{k+1}$ may only increase and does so by either $1$ or $2$ counters updating from the state $k$. Probabilities of these updates are also straightforward and one gets
\begin{align}
\nonumber
\psi_{k+1}'(t) & = 2(\psi_k(t)-\psi_{k+1}(t))\psi_{k+1}(t) + 2(\psi_k(t)-\psi_{k+1}(t))^2 
\\ 
\label{eq:mean.field}
&= 2 \psi_k(t)(\psi_k(t)-\psi_{k+1}(t)).
\end{align}
To simplify calculations, we will consider
$$
\phi_k(t)=\psi_k(t/2)
$$
which satisfy equations
$$
\phi_{k+1}'(t) = \phi_k(t)(\phi_k(t)-\phi_{k+1}(t)).
$$

With the initial conditions, the system of equations has solutions following the recursive form
\begin{equation} \label{eq:solution2}
\phi_{k+1} (t)= \int_0^t \phi_k^2(s) \exp\left\{-\int_s^t \phi_k(u)du\right\}ds.
\end{equation}
It can be seen that
$$
\phi_1(t)=1-e^{-t}
$$
but the complexity of the solutions grows quickly and this exact recursive expression does not seem to be amenable to direct analysis.

We provide below a graph (see figure \ref{fig:example})  of some of the first $\phi_k$ which demonstrates that there may be a travelling-wave solution to the system of equations.

\begin{figure}[h!] 
    \centering
    \includegraphics[width=0.6\textwidth]{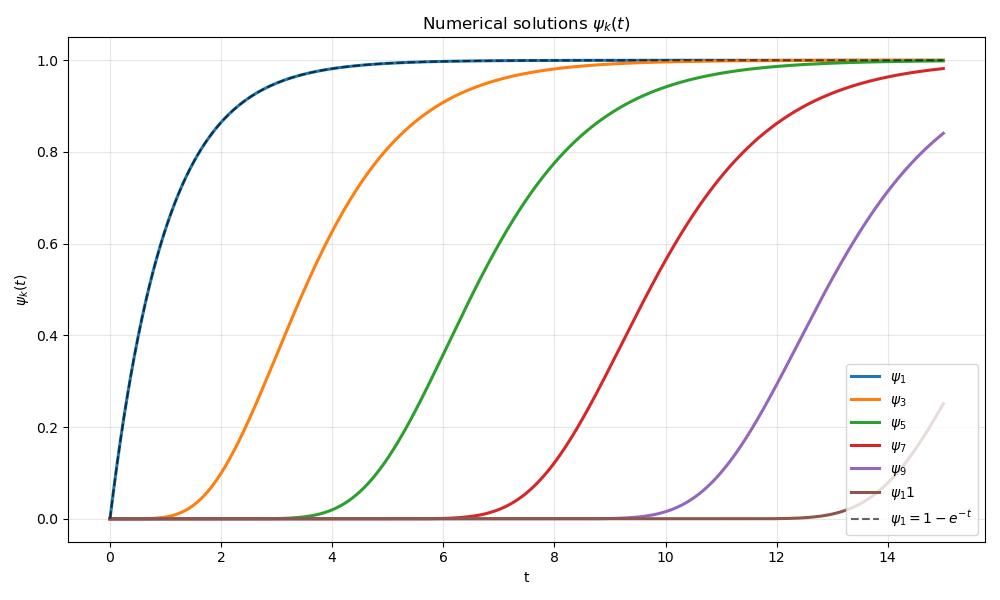} 
    \caption{Plots of $\phi_1,\phi_3,\phi_5,\phi_7, \phi_9$ and $\phi_{11}$}
    \label{fig:example}
\end{figure}

As we mentioned in the introduction, we were not able to identify the speed of this travelling-wave solution, or demonstrate its existence. We conjecture that one of the tails of the travelling-wave solution is exponential, while the other is double exponential.

\section{Preliminary computations}

 \subsection{Computation of $\E[\Delta X_{k} ]$}
In this section, we will perform the first step of our analysis and provide formulas for the expected increments of the distances between consecutive counters.

Recall that $\alpha_{i}$ denotes the number of counters at level $i$. Let also 
\[
l_{i}=\alpha_1+\cdots+\alpha_i.
\]
Then $l_i$ gives the cumulative  
number of counters at 
levels $1,2\ldots, i$. 

We will associate with  level $i$ variables $X_{l_{i-1}-1},\ldots, X_{l_i-2}$. 
We put $l_0=0$ and note 
that there numbering of random variables starts with $1$, that is there is no $X_0, X_{-1}$. This agreement is only for consistence of the notation for all levels including the first one. 

We will focus on computing
$\E[\Delta X_{l_{i}-1} ], \E[\Delta X_{l_{i}} ],\ldots, \E[\Delta X_{l_{i+1}-2} ]$ 
for each level $i+1$.

If $\alpha_1\ge 3$, then 
\begin{equation} \label{eq:level_1}
\E[\Delta X_1] = \frac{{\alpha_1 \choose 2}}{{N\choose 2}} \text{ and } \E[\Delta X_j] = 0 \text{ for all } j=2,\ldots,\alpha_1-2.
\end{equation}

For all lower levels, we have to distinguish between the following $9$ cases.
\begin{table}[h!]
\centering
\begin{tabular}{|c|c|c|c}
\hline
& Level $i+1$ & Level $i$   \\
\hline
1& $\alpha_{i+1}\ge 3$  & $\alpha_{i}\ge 3$ \\
2& $\alpha_{i+1}\ge 3$  & $\alpha_{i}=2$ \\
3& $\alpha_{i+1}\ge 3$  & $\alpha_{i}=1$  \\
\hline
4& $\alpha_{i+1}=2$  & $\alpha_{i}\ge 3$  \\
5& $\alpha_{i+1}=2$  & $\alpha_{i}= 2$  \\
6& $\alpha_{i+1}=2$  & $\alpha_{i}=1$   \\
\hline
7& $\alpha_{i+1}=1$  & $\alpha_{i}\ge 3$  \\
8& $\alpha_{i+1}=1$  & $\alpha_{i}=2$  \\
9& $\alpha_{i+1}=1$  & $\alpha_{i}=1$  \\
\hline 
\end{tabular}
\label{tab:example}
\end{table}

We now consider each case separately and determine corresponding expectations:
\begin{enumerate}[(1)]
\item 
\underline{
Case $\alpha_{i+1}\ge 3$ and $\alpha_{i}\ge 3$:
}
\begin{align*}
\E[\Delta X_{l_i-1}]
&= - \frac{l_i\alpha_{i+1}+{\alpha_{i+1}\choose 2}}{{N\choose 2}},
\\
\E[\Delta X_{l_i}]
&= 
\frac{l_i\alpha_{i+1}}{{N\choose 2}},
\\
 \E[\Delta X_{l_i+1}]
&= \frac{{\alpha_{i+1}\choose 2}}{{N\choose 2}},\\
\E[\Delta X_{l_i+j}]
&=0,  j=2,\ldots, \alpha_{i+1}-2.
\end{align*}

\item 
\underline{
Case $\alpha_{i+1}\ge 3$ and $\alpha_{i}=2$: 
}
\begin{align*}
\E[\Delta X_{l_i-1}]
&= - \frac{l_i\alpha_{i+1}+{\alpha_{i+1}\choose 2}}{{N\choose 2}}
+\frac{1}{{N\choose 2}},
\\
\E[\Delta X_{l_i}]
&= 
\frac{l_i\alpha_{i+1}}{{N\choose 2}},
\\
 \E[\Delta X_{l_i+1}]
&= \frac{{\alpha_{i+1}\choose 2}}{{N\choose 2}},\\
\E[\Delta X_{l_i+j}]
&=0,  j=2,\ldots, \alpha_{i+1}-2.
\end{align*}

\item 
\underline{
Case $\alpha_{i+1}\ge 3$ and $\alpha_{i}=1$: 
}
\begin{align*}
\E[\Delta X_{l_i-1}]
&= - \frac{l_i\alpha_{i+1}+{\alpha_{i+1}\choose 2}}{{N\choose 2}}
+\frac{l_{i-1}}{{N\choose 2}},
\\
\E[\Delta X_{l_i}]
&= 
\frac{l_i\alpha_{i+1}}{{N\choose 2}},
\\
 \E[\Delta X_{l_i+1}]
&= \frac{{\alpha_{i+1}\choose 2}}{{N\choose 2}},\\
\E[\Delta X_{l_i+j}]
&=0,  j=2,\ldots, \alpha_{i+1}-2.
\end{align*}

\item 
\underline{
Case $\alpha_{i+1}=2$ and $\alpha_{i}\ge 3$: 
}
\begin{align*}
\E[\Delta X_{l_i-1} ] &= 
-\frac{2l_i+1}{{N \choose 2}}
=-\frac{2l_{i}+1}{{N \choose 2}},\\
\E[\Delta X_{l_i} ] &=
\frac{2l_i}{{N \choose 2}}.
\end{align*}

\item 
\underline{
Case $\alpha_{i+1}=2$ and $\alpha_{i}=2$: 
}
\begin{align*}
\E[\Delta X_{l_i-1} ] &= 
-\frac{2l_i+1}{{N \choose 2}}
+\frac{1}{{N \choose 2}}
=-\frac{2l_{i}}{{N \choose 2}},\\
\E[\Delta X_{l_i} ] &=
\frac{2l_i}{{N \choose 2}}.
\end{align*}

\item 
\underline{
Case $\alpha_{i+1}=2$ and $\alpha_{i}=1$: 
}
\begin{align*}
\E[\Delta X_{l_i-1} ] &= 
-\frac{2l_i+1}{{N \choose 2}}
+\frac{l_{i-1}}{{N \choose 2}}
=-\frac{l_{i}+2}{{N \choose 2}},\\
\E[\Delta X_{l_i} ] &=
\frac{2l_i}{{N \choose 2}}.
\end{align*}

\item 
\underline{
Case $\alpha_{i+1}=1$ and  $\alpha_{i}\ge 3$: 
}
\[
\E[\Delta X_{l_i-1} ] = 
-\frac{l_i}{{N \choose 2}}
.
\]

\item 
\underline{
Case $\alpha_{i+1}=1$ and  $\alpha_{i}=2$: 
}
\[
\E[\Delta X_{l_i-1} ] = 
-\frac{l_i}{{N \choose 2}}
+\frac{1}{{N \choose 2}}
=-\frac{l_i-1}{{N \choose 2}}.
\]

\item  
\underline{
Case $\alpha_{i+1}=\alpha_i=1$: 
}
\[
\E[\Delta X_{l_i-1} ] = 
-\frac{l_i}{{N \choose 2}}
+\frac{l_{i-1}}{{N \choose 2}}
=-\frac{1}{{N \choose 2}}.
\] 
\end{enumerate}

We summarise the computations of the expected changes $\E[\Delta X_i]$ in Table~\ref{table:expectations} below. 
Each row corresponds to a particular combination of $(\alpha_{i+1},\alpha_i)$, 
and lists the non-zero expectations at positions $l_i-1$, $l_i$, and $l_i+1$. 
All other entries are zero.

\begin{table}[h!]
\centering
\renewcommand{\arraystretch}{1.4}
\begin{tabular}{|c|c|c|c|c|}
\hline
Case & $(\alpha_{i+1},\alpha_i)$ & $\E[\Delta X_{l_i-1}]$ & $\E[\Delta X_{l_i}]$ & $\E[\Delta X_{l_i+1}]$ \\
\hline
(1) & $(\ge 3,\ge 3)$ 
& $-\tfrac{l_i\alpha_{i+1}+{\alpha_{i+1}\choose 2}}{{N\choose 2}}$ 
& $\tfrac{l_i\alpha_{i+1}}{{N\choose 2}}$ 
& $\tfrac{{\alpha_{i+1}\choose 2}}{{N\choose 2}}$ \\
\hline
(2) & $(\ge 3,2)$ 
& $-\tfrac{l_i\alpha_{i+1}+{\alpha_{i+1}\choose 2}}{{N\choose 2}}
+\tfrac{1}{{N\choose 2}}$ 
& $\tfrac{l_i\alpha_{i+1}}{{N\choose 2}}$ 
& $\tfrac{{\alpha_{i+1}\choose 2}}{{N\choose 2}}$ \\
\hline
(3) & $(\ge 3,1)$ 
& $- \frac{l_i\alpha_{i+1}+{\alpha_{i+1}\choose 2}}{{N\choose 2}}
+\frac{l_{i-1}}{{N\choose 2}}$ 
& $\tfrac{l_i\alpha_{i+1}}{{N\choose 2}}$ 
& $\tfrac{{\alpha_{i+1}\choose 2}}{{N\choose 2}}$ \\
\hline
(4) & $(2,\ge 3)$ 
& $-\tfrac{2l_i+1}{{N\choose 2}}$ 
& $\tfrac{2l_i}{{N\choose 2}}$ & -- \\
\hline
(5) & $(2,2)$ 
& $-\tfrac{2l_i}{{N\choose 2}}$ 
& $\tfrac{2l_i}{{N\choose 2}}$ & -- \\
\hline
(6) & $(2,1)$ 
& $-\tfrac{l_i+2}{{N\choose 2}}$ 
& $\tfrac{2l_i}{{N\choose 2}}$ & -- \\
\hline
(7) & $(1,\ge 3)$ 
& $-\tfrac{l_i}{{N\choose 2}}$ & -- & -- \\
\hline
(8) & $(1,2)$ 
& $-\tfrac{l_i-1}{{N\choose 2}}$ & -- & -- \\
\hline
(9) & $(1,1)$ 
& $-\tfrac{1}{{N\choose 2}}$ & -- & -- \\
\hline
\end{tabular}
\caption{Summary of expectations $\E[\Delta X_i]$ for all cases of $(\alpha_{i+1},\alpha_i)$. 
Entries not shown explicitly are zero, i.e., for cases (7), (8), and (9), 
$\E[\Delta X_{l_i+j}] = 0$ for $j \ge 2$.}
\label{table:expectations}
\end{table}

\subsection{Computation of values of test functions for each configuration }

With an arbitrary test function $h$ and and arbitrary constant $v$, for any given configuration 
$\alpha=(\alpha_1,\alpha_2,\ldots,\alpha_M)$ 
we will consider 
\[
L_\alpha h:=
\sum_{k=1}^{N-2} h(k)\E[\Delta X_k]
+v-1- \sum_{i=1}^M\frac{{\alpha_{i+1}\choose 2}}{{N\choose 2}}. 
\] 

Recall again that
\[
1+\sum_{i=1}^M\frac{{\alpha_{i+1}\choose 2}}{{N\choose 2}}
\]
is the expected number of 
counters  updated in a given configuration. 

With the convention $\Delta X_{-1}=\Delta X_0=0$, we can decompose $L_\alpha h$ 
according to contributions of different levels as follows:
\[
L_\alpha h
 =v-1
+
\sum_{i=0}^{M-1} 
\left(
\sum_{j=0}^{\alpha_{i+1}-1}
h(l_{i}+j-1)
\E[\Delta X_{l_{i}+j-1}]
-\frac{{\alpha_{i+1}\choose 2}}{{N\choose 2}}
\right).
\]
Denote the contribution 
of level $i+1$ by 
\begin{multline}\label{expect.level}
E(l_i,\alpha_{i},\alpha_{i+1}):=
\sum_{j=0}^{\alpha_{i+1}-1}
h(l_{i}+j-1)
\E[\Delta X_{l_{i}+j-1}]
-\frac{{\alpha_{i+1}\choose 2}}{{N\choose 2}}\\
=
\sum_{j=0}^{\alpha_{i+1}-1}
h(l_{i}+j-1)
\E[\Delta X_{l_{i}+j-1}]
-\frac{\alpha_{i+1}(\alpha_{i+1}-1)}{N(N-1)}.
\end{multline}
We have included the parameter
$\alpha_{i}$ as 
the expectations $\E[\Delta X_{l_i+j-1}]$ depend on it directly for each level 
$i+1$. 

We will next compute the function in~\eqref{expect.level} 
for various combinations of $\alpha_i$ 
and $\alpha_{i+1}$. 

We will treat separately the first level. 
We define $\alpha_0=0,l_0=0$. 
If $\alpha_1=2$ then
\begin{equation} \label{eq:0_0_2}
E(0,0,2) = -\frac{1}{{N\choose 2}}
=-\frac{2}{N(N-1)}. 
\end{equation}
If  $\alpha_1 \ge 3$ then, by \eqref{eq:level_1},
\begin{align}
E(0,0,\alpha_1) &= h(1)\,\E[\Delta X_1]
- \frac{\alpha_1 (\alpha_1 - 1)}{N(N-1)} 
= \frac{\alpha_1 (\alpha_1 - 1)}{N(N-1)} \big(h(1) - 1\big). \label{eq:0_0_3}
\end{align}

In order to make the following computations more transparent, we rewrite 
 Table~\ref{table:expectations} 
 in the following equivalent form, Table~\ref{table:expectations.2}. 
 The advantage of this new form is that  
 the differences between cases are highlighted.

\begin{table}[h!]
\centering
\renewcommand{\arraystretch}{1.4}
\begin{tabular}{|c|c|c|c|c|}
\hline
Case & $(\alpha_{i+1},\alpha_i)$ & $\E[\Delta X_{l_i-1}]$ & $\E[\Delta X_{l_i}]$ & $\E[\Delta X_{l_i+1}]$ \\
\hline
(1) & $(\ge 3,\ge 3)$ 
& $-\tfrac{2l_i\alpha_{i+1}+
\alpha_{i+1}(\alpha_{i+1}-1)}
{N(N-1)}$ 
& $\tfrac{2l_i\alpha_{i+1}}{N(N-1)}$ 
& $\tfrac{\alpha_{i+1}(\alpha_{i+1}-1)}{N(N-1)}$ \\
\hline
(2) & $(\ge 3,2)$ 
& $-\tfrac{2l_i\alpha_{i+1}+
\alpha_{i+1}(\alpha_{i+1}-1)-2}
{N(N-1)}$ 
& $\tfrac{2l_i\alpha_{i+1}}{N(N-1)}$ 
& $\tfrac{\alpha_{i+1}(\alpha_{i+1}-1)}{N(N-1)}$ \\
\hline
(3) & $(\ge 3,1)$ 
& $-\tfrac{2l_i\alpha_{i+1}+
\alpha_{i+1}(\alpha_{i+1}-1)-2l_{i-1}}
{N(N-1)}$ 
& $\tfrac{2l_i\alpha_{i+1}}{N(N-1)}$ 
& $\tfrac{\alpha_{i+1}(\alpha_{i+1}-1)}{N(N-1)}$ \\
\hline
(4) & $(2,\ge 3)$ 
& $-\tfrac{2l_i\alpha_{i+1}+
\alpha_{i+1}(\alpha_{i+1}-1)}
{N(N-1)}$ 
& $\tfrac{2l_i\alpha_{i+1}}{N(N-1)}$ & 
$\tfrac{\alpha_{i+1}(\alpha_{i+1}-1)}{N(N-1)}
-\boxed{\tfrac{2}{N(N-1)}}$
\\
\hline
(5) & $(2,2)$ 
& $-\tfrac{l_i\alpha_{i+1}+
\alpha_{i+1}(\alpha_{i+1}-1)-2}
{N(N-1)}$ 
& $\tfrac{2l_i\alpha_{i+1}}{N(N-1)}$ & $\tfrac{\alpha_{i+1}(\alpha_{i+1}-1)}{N(N-1)}
-\boxed{\tfrac{2}{N(N-1)}}$ \\
\hline
(6) & $(2,1)$ 
& $-\tfrac{2l_i\alpha_{i+1}+
\alpha_{i+1}(\alpha_{i+1}-1)-2l_{i-1}}
{N(N-1)}$ 
& $\tfrac{2l_i\alpha_{i+1}}{N(N-1)}$ & $\tfrac{\alpha_{i+1}(\alpha_{i+1}-1)}{N(N-1)}
-\boxed{\tfrac{2}{N(N-1)}}$ \\
\hline
(7) & $(1,\ge 3)$ 
& $-\tfrac{2l_i\alpha_{i+1}+
\alpha_{i+1}(\alpha_{i+1}-1)}
{N(N-1)}$ & 
$\tfrac{2l_i\alpha_{i+1}}{N(N-1)}
-\boxed{\tfrac{2l_i}{N(N-1)}}
$
& $\tfrac{\alpha_{i+1}(\alpha_{i+1}-1)}{N(N-1)}$ \\
\hline
(8) & $(1,2)$ 
& $-\tfrac{2l_i\alpha_{i+1}+
\alpha_{i+1}(\alpha_{i+1}-1)-2}
{N(N-1)}$ & $\tfrac{2l_i\alpha_{i+1}}{N(N-1)}
-\boxed{\tfrac{2l_i}{N(N-1)}}$ & $\tfrac{\alpha_{i+1}(\alpha_{i+1}-1)}{N(N-1)}$ \\
\hline
(9) & $(1,1)$ 
& $-\tfrac{2l_i\alpha_{i+1}+
\alpha_{i+1}(\alpha_{i+1}-1)-2l_{i-1}}
{N(N-1)}$ & $\tfrac{2l_i\alpha_{i+1}}{N(N-1)}
-\boxed{\tfrac{2l_i}{N(N-1)}}$ & $\tfrac{\alpha_{i+1}(\alpha_{i+1}-1)}{N(N-1)}$ \\
\hline
\end{tabular}
\caption{Equivalent form of Table \ref{table:expectations}}
\label{table:expectations.2}
\end{table}

We will compute $E(l_i,\alpha_i,\alpha_{i+1})$
separately for each of the $9$ cases. 
We can see from Table~\ref{table:expectations.2} 
that the main cases are (1)--(3) and 
other cases will have extra contributions from boxed summands.

\begin{enumerate}[(1)]
    \item 
    \underline{Case $\alpha_{i+1}\ge 3$ and $\alpha_{i}\ge 3$.}
Note that 
\begin{align*}
& E(l_i,\alpha_i,\alpha_{i+1}) 
=: E^{(1)}(l_i, \alpha_{i+1})
\\&= 
  \sum_{m=-1}^{\alpha_{i+1}-2}
  h(l_i+m)\E[\Delta X_{l_i+m}] -\frac{\alpha_{i+1}(\alpha_{i+1}-1)}{N(N-1)}
  \\ 
  &= h(l_i-1)\E[\Delta X_{l_i-1}]
  +
  h(l_i)\E[\Delta X_{l_i}]
  +
  h(l_i+1)\E[\Delta X_{l_i+1}]
  -\frac{\alpha_{i+1}(\alpha_{i+1}-1)}{N(N-1)}\\
  & =: E^*(l_i,\alpha_{i+1}),  
\end{align*}
where 
\[
E^*(l,\alpha) 
=\frac{-\alpha(2l+\alpha-1)h(l-1)
+2l\alpha h(l)
+\alpha(\alpha-1)h(l+1)
-\alpha(\alpha-1)
}{N(N-1)}.
\]
Denote $\Delta h(l)=h(l+1)-h(l)$. 
Then,
\[
h(l+1) = h(l)+\Delta h(l) 
\text{ and } 
h(l-1) = h(l)-\Delta h(l-1).
\]
We can then rewrite
\begin{equation}\label{e.star}
E^{*}(l,\alpha)
=\frac{\alpha(2l+\alpha-1)\Delta h(l-1)
+
\alpha(\alpha-1)
\Delta h(l)-\alpha(\alpha-1)
}{N(N-1)}.     
\end{equation}




    \item 
    \underline{Case $\alpha_{i+1}\ge 3$ and $\alpha_{i}=2$.}
In comparison with case (1), we have
\begin{align*}
E(l_i,2,\alpha_{i+1}) =: 
E^{(2)}(l_i, \alpha_{i+1})
=
E^*(l_i, \alpha_{i+1})
+\frac{2h(l_{i}-1)}{N(N-1)}.
\end{align*}

    \item 
    \underline{Case $\alpha_{i+1}\ge 3$ and $\alpha_{i}=1$.}
In comparison with case (1), we have
\begin{align*}
E(l_i,1,\alpha_{i+1}) &=: 
E^{(3)}(l_i, \alpha_{i+1})\\ 
&=
E^*(l_i, \alpha_{i+1})
+\frac{2(l_{i}-1)h(l_{i}-1)}{N(N-1)},
\end{align*}
since $l_i-1=l_{i-1}$ when $\alpha_i=1$.

    \item 
    \underline{Case $\alpha_{i+1}=2$ and $\alpha_{i}\ge 3$.}
Now, Table~\ref{table:expectations.2} 
shows that this is almost 
case (1) with one extra contribution, that is,  
\[
E(l_i,\alpha_i,2) 
=
E^{(1)}(l_i, \alpha_{i+1})
-\frac{2}{N(N-1)} h(l_{i}+1).
\]
    \item 
    \underline{Case $\alpha_{i+1}=2$ and $\alpha_{i}=2$.}

  Table~\ref{table:expectations.2} 
shows that this is  
case (2) with one extra contribution, that is,
\[
E(l_i,\alpha_i,2) 
=
E^{(2)}(l_i, \alpha_{i+1})
-\frac{2}{N(N-1)} h(l_{i}+1).
\]  

    \item 
    \underline{Case $\alpha_{i+1}=2$ and $\alpha_{i}=1$.}
Table~\ref{table:expectations.2} 
shows that this is  
case (3) with one extra contribution, that is, 
\[
E(l_i,\alpha_i,2) 
=
E^{(3)}(l_i, \alpha_{i+1})
-\frac{2}{N(N-1)} h(l_{i}+1).
\]  

    \item 
    \underline{Case $\alpha_{i+1}=1$ and $\alpha_{i}\ge 3$.}
Table~\ref{table:expectations.2} 
shows that this is  
case (1) with one extra contribution, that is,
\[
E(l_i,\alpha_i,1) 
=
E^{(1)}(l_i, \alpha_{i+1})
-\frac{2l_i h(l_i)}{N(N-1)}.
\]

    \item 
    \underline{Case $\alpha_{i+1}=1$ and $\alpha_{i}=2$.}
Table~\ref{table:expectations.2} 
shows that this is  
case (2) with one extra contribution, that is,
\[
E(l_i,\alpha_i,1) 
=
E^{(2)}(l_i, \alpha_{i+1})
-\frac{2l_i h(l_i)}{N(N-1)}.
\]  

    \item 
    \underline{Case $\alpha_{i+1}=\alpha_i=1$.}
Table~\ref{table:expectations.2} 
shows that this is  
case (3) with one extra contribution, that is,
\[
E(l_i,\alpha_i,1) 
=
E^{(3)}(l_i, \alpha_{i+1})
-\frac{2l_i h(l_i)}{N(N-1)}.
\]  

\end{enumerate}
We can summarise all cases in the table below. 
\begin{table}[h!]
\centering
\renewcommand{\arraystretch}{1.4}
\begin{tabular}{|c|c|c|}
\hline
Case & $(\alpha_i,\alpha_{i+1})$ & $E(l_i,\alpha_i,\alpha_{i+1})$ \\
\hline
(1) & $\alpha_i \ge 3,\;\alpha_{i+1} \ge 3$ &
$\displaystyle
E^*(l_i,\alpha_{i+1})
$ \\
\hline
(2) & $\alpha_i=2,\;\alpha_{i+1}\ge 3$ &
$\displaystyle 
E^*(l_i,\alpha_{i+1})+\frac{2h(l_i-1)}{N(N-1)}$ \\
\hline
(3) & $\alpha_i=1,\;\alpha_{i+1}\ge 3$ &
$\displaystyle 
E^*(l_i,\alpha_{i+1})
+\frac{2(l_i-1)h(l_i-1)}{N(N-1)}$ \\
\hline
(4) & $\alpha_i \ge 3,\;\alpha_{i+1}=2$ &
$\displaystyle 
E^{(1)}(l_i,\alpha_{i+1})-\frac{2h(l_i+1)}{N(N-1)}$ \\
\hline
(5) & $\alpha_i=2,\;\alpha_{i+1}=2$ &
$\displaystyle 
E^{(2)}(l_i,\alpha_{i+1})-\frac{2h(l_i+1)}{N(N-1)}$ \\
\hline
(6) & $\alpha_i=1,\;\alpha_{i+1}=2$ &
$\displaystyle 
E^{(3)}(l_i,\alpha_{i+1})-\frac{2h(l_i+1)}{N(N-1)}$ \\
\hline
(7) & $\alpha_i \ge 3,\;\alpha_{i+1}=1$ &
$\displaystyle 
E^{(1)}(l_i,\alpha_{i+1})-\frac{2l_i h(l_i)}{N(N-1)}$ \\
\hline
(8) & $\alpha_i=2,\;\alpha_{i+1}=1$ &
$\displaystyle 
E^{(2)}(l_i,\alpha_{i+1})-\frac{2l_i h(l_i)}{N(N-1)}$ \\
\hline
(9) & $\alpha_i=1,\;\alpha_{i+1}=1$ &
$\displaystyle 
E^{(3)}(l_i,\alpha_{i+1})-\frac{2l_i h(l_i)}{N(N-1)}$ \\
\hline
\end{tabular}
\label{table.contributions}
\end{table}

We can combine all the cases above into the following formula:
\begin{multline}
\label{e.l.alpha.beta}
E(l_i,\alpha_i,\alpha_{i+1})
=
E^*(l_i,\alpha_{i+1})
+\frac{2 h(l_i-1)}{N(N-1)}I(\alpha_i=2)
+\frac{2 (l_i-1)h(l_i-1)}{N(N-1)}I(\alpha_i=1)\\
-
\frac{2 h(l_i+1)}{N(N-1)}I(\alpha_{i+1}=2)
-\frac{2l_ih(l_i)}{N(N-1)}I(\alpha_{i+1}=1). 
\end{multline}
Recall that 
\[
L_\alpha h
=
v-1+
E(0,0,\alpha_1)+
\sum_{i=1}^{M-1}
E(l_i,\alpha_i,\alpha_{i+1}).
\]
Utilising~\eqref{e.l.alpha.beta}, we further rewrite
\begin{multline*}
\sum_{i=1}^{M-1}
E(l_i,\alpha_i,\alpha_{i+1})
=
\sum_{i=1}^{M-1}
E^*(l_i,\alpha_{i+1})\\
+
\frac{2}{N(N-1)}
\Biggl(
\sum_{i=1}^{M-1}
h(l_i-1)I(\alpha_i=2)
-
\sum_{i=1}^{M-1}
h(l_i+1)I(\alpha_{i+1}=2)\\
+
\sum_{i=1}^{M-1}
(l_i-1)h(l_i-1)I(\alpha_i=1)
-
\sum_{i=1}^{M-1}
l_ih(l_i)I(\alpha_{i+1}=1)
\Biggr).
\end{multline*}
We can observe now that
\begin{multline*}
    \sum_{i=1}^{M-1}
h(l_i-1)I(\alpha_i=2)
-
\sum_{i=1}^{M-1}
h(l_i+1)I(\alpha_{i+1}=2)\\
=
\sum_{i=1}^{M-1}
h(l_i-1)I(\alpha_i=2)
-
\sum_{i=1}^{M-1}
h(l_{i+1}-1)I(\alpha_{i+1}=2)\\
=h(l_1-1)I(\alpha_1=2)
-h(l_M-1)I(\alpha_M=2)
\end{multline*}
and 
\begin{multline*}
    \sum_{i=1}^{M-1}
(l_i-1)h(l_i-1)I(\alpha_i=1)
-
\sum_{i=1}^{M-1}
l_ih(l_i)I(\alpha_{i+1}=1)
\\
=
\sum_{i=1}^{M-1}
(l_i-1)h(l_i-1)I(\alpha_i=1)
-
\sum_{i=1}^{M-1}
(l_{i+1}-1)h(l_{i+1}-1)I(\alpha_{i+1}=1)\\
=
-(l_M-1)h(l_M-1)I(\alpha_M=1),
\end{multline*}
as $I(\alpha_1=1)=0$. Therefore, 
\begin{multline*}
    L_\alpha h 
    =v-1+E(0,0,\alpha_1)+
    \sum_{i=1}^{M-1}
E^*(l_i,\alpha_{i+1})
+\frac{2}{N(N-1)}
\Bigl(
h(l_1-1)I(\alpha_1=2)
\\
-h(l_M-1)I(\alpha_M=2)
-(l_M-1)h(l_M-1)I(\alpha_M=1)
\Bigr).
\end{multline*}
Next, combining \eqref{eq:0_0_2} and \eqref{eq:0_0_3}, we write
\begin{align*}
E(0,0,\alpha_1)
+\frac{2}{N(N-1)}
h(\alpha_1-1)
I(\alpha_1=2)
&= \frac{\alpha_1(\alpha_1-1)}{N(N-1)}(h(1)-1).\\
\end{align*}
Now, note that 
\begin{align*}
E^*(0,\alpha_1)
&=
\frac{\alpha(\alpha-1)}{N(N-1)}
\left(
\Delta h(-1)
+\Delta h(0)-1
\right)\\
&
=\frac{\alpha(\alpha-1)}{N(N-1)}
\left(
h(1)-h(-1)-1
\right),
\end{align*}
so, imposing the assumption $h(-1)=0$, we obtain 
\[
E(0,0,\alpha_1)
+\frac{2}{N(N-1)}
h(\alpha_1-1)
I(\alpha_1=2)
=E^*(0,\alpha_1).
\]
Hence, 
recalling that $l_0=0,$  
we have a 
slightly simpler expression 
\begin{multline}\label{sum.simplified}
L_\alpha h 
    =v-1
    + 
    \sum_{i=0}^{M-1}
E^*(l_i,\alpha_{i+1})\\
-\frac{2}{N(N-1)}
\Bigl(
h(l_M-1)I(\alpha_M=2)
+(l_M-1)h(l_M-1)I(\alpha_M=1)
\Bigr).
\end{multline}

At the moment we have not made any assumptions about $h$ apart from it 
being defined for all integers on $[-1,N]$ and 
\[
h(-1)=0.
\]

\section{Proof of Theorem~\ref{thm:positive.recurrent}}

We provide two proofs of positive recurrence, using a quadratic and an exponential Lyapunov function. As is standard convention, we will use $\E_x$ to denote expectations conditioned on $X(0)=x$.

\subsection{Proof with quadratic Lyapunov function}

Consider  the following Lyapunov function 
\[
L(x)=x_1^2+\cdots+x_{N-2}^2
\] 
for the vector $x=(x_1,\ldots,x_{N-2})$.  Then, there exists a constant
$R$ such that 
$$
\E_x L(X(1))-L(x) \le -1
$$
for all $x$ with $\max  x_i>R$. This is sufficient for the positive recurrence, thanks to the well-known Foster-Lyapunov criteria.

To prove the bound above, we denote $\Delta X_i =X_i-x_i$. Then,
\begin{align*}
  \E_x L(X_1)-L(x) = 2\sum_{i=1}^{N-2}  x_i \E_x  \Delta X_i +\E_x\sum_{i=1}^{N-2} 
  \Delta X_i^2. 
\end{align*}
Next note that 
$$
\sum_{i=1}^N \E_x
  \Delta X_i^2 \le 4,
$$
as at most 2 points can  be selected 
and moved up, thus changing the distances at most in $4$ cases.   
When $x_i\ge 1$ we can observe that 
\begin{equation}\label{eq0}
\E_x \Delta X_i\le \frac{-1}{{N\choose 2}},
\end{equation}
which can be seen from the second column in Table~\ref{table:expectations}.  

Once we have~\eqref{eq0} 
we obtain  immediately the statement as 
$$
\E_x L(X_1)-L(x) \le 2 \max x_i \frac{-1}{{N\choose 2}} +4\le -1
$$
once 
\[
\max x_i \ge \frac{5}{4} N(N+1).
\]


\subsection{Proof with exponential Lyapunov function}
Let, for some $r\in(0,1)$,  
\[
L(x)=e^{rx_1}+\cdots+e^{rx_{N-2}}
\] 
for the vector $x=(x_1,\ldots,x_{N-2})$.
Then, using inequality $e^x\le 1+x+x^2$ valid for $x\le 1$, we obtain 
\begin{align*}
  \E_x L(X_1)-L(x)&= 
  \sum_{i=1}^{N-2} e^{rx_i}\E_x e^{r\Delta X_i}-L(x)\\
  &\le 
  \sum_{i=1}^{N-2} 
e^{r x_i}
\left(
1+r\E_x \Delta X_i+r^2\E_x (\Delta X_i)^2\right)-L(x)\\
&=
\sum_{i=1}^{N-2} 
(e^{r x_i}-1)
\left(
r\E_x \Delta X_i+r^2\E_x (\Delta X_i)^2\right)\\
&
\phantom{XX}+
\sum_{i=1}^{N-2} 
\left(
r\E_x \Delta X_i+r^2\E_x (\Delta X_i)^2\right). 
\end{align*}
Note 
from computations of the expectation 
that when $x_i>0$, see Table~\ref{table:expectations},
\[
r \E_x\Delta X_i
+r^2\E_x(\Delta X_i)^2
\le 2\frac{-r+r^2}{N(N-1)}
\]
and is negative 
when 
\[
r=\frac{1}{2}.
\]
Therefore, 
\[
\E_x L(X_1)-L(x)
\le -\frac{1}{2N(N-1)}
\sum_{i=1}^{N-2}
(e^{\frac12 x_i}-1)
+2.
\]
The RHS is bounded from above by $-1$ if, for example, 
\[
-\frac{1}{N(N-1)} (e^{\frac12\max x_i}-1)
+3<0.
\]
This holds when 
when 
\[
\max x_i
>
2
\ln\left(16N^2\right).
\]

\section{Monotonicity property of configurations}\label{sec:reduction}

We start with 
a general quadratic function 
\begin{equation}\label{defn.f.quadratic}
h(x)=
A(x^2-1)+B(x+1).
\end{equation}
This particular form ensures 
that condition $h(-1)=0$ holds.

There are $2^{N-2}$ possible configurations, which is a 
large number to analyse directly. In this section 
we will discuss how to reduce the problem to the analysis of several  cases. 
First, a definition: we say that 
configuration $\hat \alpha$  is worse than $\alpha$ if 
$L_\alpha (h)\ge L_{\hat\alpha} (h)$.  
So we will aim to find several configurations, which 
will be worse than all other configurations and then our task will be to analyse these configurations.

We will use two devices to obtain worse configurations: merging of levels, when  all counters from one level are moved to another level; 
and rebalancing, when counters in two levels will be moved to ensure that they have equal number of counters or differ by at most  $1$ counter. 

Recall that $M$ denotes the number of levels. We will consider levels $k$ and $k-1$ that are neighbours. When $k=M$ we exclude the situation $\alpha_M\in\{1,2\}$ from the considerations, we keep the last level as it is. 

We obtain a new configuration $\hat \alpha$ from $\alpha$ 
by moving $\beta$ counters 
from level $k-1$ to $k$. All other counters stay at the same place. If $\beta $ is negative, this means 
that we are moving $-\beta$ counters 
from level $k$ to level $k-1$. 
If $\beta=\alpha_{k-1}$ we have moved all counters from level $k-1$ to $k$, which corresponds to merging of levels. For $\beta\in(-\alpha_{k},\alpha_{k-1})$ we have rebalancing of levels. 
The change in $L_\alpha f$ is given by, using~\eqref{sum.simplified},
\begin{multline*}
    L_\alpha h - L_{\hat \alpha} h 
    =
    D_k(\beta):= E^*(l_{k-1},\alpha_k) +E^*(l_{k-2},\alpha_{k-1}) \\-E^*(l_{k-1}-\beta,\alpha_k+\beta) -E^*(l_{k-2},\alpha_{k-1}-\beta).
\end{multline*}
To analyse this difference we prove an auxiliary lemma first.  Let operator $\mathcal E_a$ be given by 
\begin{align*}
\mathcal E_{c} h &= -\alpha(2l+c-1)hl-1)
+2lch(l)
+c(c-1)h(l+1)\\
&=
c(2l+c-1)\Delta h(l-1)
+
c(c-1)
\Delta h(l),
\end{align*}
so that 
\[
E^*(l,c) 
=\frac{\mathcal E_{c} h(l)
-c(c-1)
}{N(N-1)}. 
\]
\begin{lemma}
Let $P_n(x) = {x\choose n}$ be 
a binomial coefficient. 
For arbitrary $l$ 
let 
\[
Q_n(a,b):=
\mathcal E_{b}P_n(l+a)
+
\mathcal E_{a}P_n(l)
-
\mathcal E_{b+\beta}P_n(l+a-\beta)
-\mathcal E_{a-\beta}P_n(l).
\]
Then,
\begin{align*}
Q_0(a,b) &=0\\
Q_1(a,b) &= 2\beta(a-b-\beta) \\ 
Q_2(a,b) &= -2\beta(a-b-\beta)(a+b-1).
\end{align*}
\end{lemma}
\begin{proof}
If $n=0$, 
then $P_0(x)=1$, and we see
immediately from $\Delta P_0(l)=0$ that 
\[
Q_0(a,b) = 0. 
\]
Next, note that 
\begin{align*}
\Delta P_n(x)
&=P_n(x+1)
-P_n(x)
={x+1 \choose n}
-{x \choose n}\\
&={x \choose n-1} = P_{n-1}(x),
\end{align*}
which gives us 
\begin{align*}
\mathcal E_\alpha  P_n(x)
=\alpha(2x+\alpha-1)P_{n-1}(x-1)
+\alpha(\alpha-1)P_{n-1}(x).
\end{align*}
Hence, for $n=1$ (so that $P_1(x)=x$ and $P_0(x)=1$), 
we have 
\[
\mathcal E_\alpha  P_1(x)
=2\alpha(x+\alpha-1).
\]
Therefore, 
\begin{align*}
Q_1(a,b)
&=
2b(l+a+b-1)
+2a(l+a-1)
-2(b+\beta)(l+a+b-1)\\
&\phantom{XX}-2(a-\beta)(l+a-\beta-1)
=2\beta(a-b-\beta).
\end{align*}
For $n=2$, 
\begin{align*}
\mathcal E_\alpha  P_2(x)
&=
\alpha(2x+\alpha-1)(x-1)
+\alpha(\alpha-1)x\\
&=
2\alpha(x+\alpha-1)(x-1)
+\alpha(\alpha-1).
\end{align*}
Therefore, 
\begin{align*}
Q_2(a,b)
&=2b(l+a+b-1)(l+a-1)+b(b-1)
+2a(l+a-1)(l-1)+a(a-1)
\\
&\phantom{XX}-
2(b+\beta)(l+a+b-1)(l+a-\beta-1)-(b+\beta)(b+\beta-1)\\
&\phantom{XX}-
2(a-\beta)(l+a-\beta-1)(l-1)-(a-\beta)(a-\beta-1)\\
&=-2\beta(a-b-\beta)(a+b-1)
. \qedhere
\end{align*}
\end{proof}
To apply  this lemma to~\eqref{defn.f.quadratic} observe first 
that 
\[
h(x) = 2AP_2(x)+(B+A)P_1(x)+2B
\]
to obtain 
\begin{align}
\nonumber
D_k(\beta)&=\frac{2AQ_2(\alpha_{k-1},\alpha_k)+(B+A)Q_1(\alpha_{k-1},\alpha_k)
}{N(N-1)}\\
\nonumber 
&\phantom{XX}-
\frac{\alpha_{k}(\alpha_{k}-1)+\alpha_{k-1}(\alpha_{k-1}-1)-(\alpha_{k}+\beta)(\alpha_{k}+\beta-1)}{N(N-1)}
\\
\nonumber
&\phantom{XX}-\frac{(\alpha_{k-1}-\beta)(\alpha_{k-1}-\beta-1)}{N(N-1)}\\\nonumber
&
=
\frac{-2\beta(\alpha_{k-1}-\alpha_{k}-\beta)\left(2A(\alpha_{k-1}+\alpha_{k}-1)-B-A\right)}{N(N-1)}
-\frac{2\beta(\alpha_{k-1}-\alpha_{k}-\beta)}{N(N-1)}
\\
&=
-2\beta(\alpha_{k-1}-\alpha_{k}-\beta)
\frac{2A(\alpha_k+\alpha_{k-1})-3A-B+1}{N(N-1)}.
\label{Dk.beta}
\end{align}

\subsection{Merging levels $k-1$ and $k$ for $k\ge 2$ }\label{subsec:merge.not.1} 

We assume that  
$k\in\{2,\ldots,M-1\} $ or $k=M,\alpha_M\ge 3$.

Specifying~\eqref{Dk.beta} to the case of merging, that 
is when $\beta=\alpha_{k-1}$, 
we obtain 
\begin{align*}
D_k(\alpha_{k-1})
&=
2\alpha_k\alpha_{k-1}
\frac{2A(\alpha_k+\alpha_{k-1})-3A-B+1}{N(N-1)}. 
\end{align*}
Now  choose  
\[
A=-\frac{1}{2(N-2)}.
\]
Observe that 
\[
\frac{1-3A-B}{2A}
=-\frac{3}{2}+\frac{1-B}{2A}
=-\frac{3}{2}-(N-2)(1-B). 
\]
Then $D_k(\alpha_{k-1})\ge 0$ provided 
\begin{equation}\label{eq:merge}
\alpha_{k-1}+\alpha_k\le N(B):=(N-2)(1-B)+\frac{3}{2}. 
\end{equation}
Therefore, when we merge levels 
$k-1$ and $k$ satisfying~\eqref{eq:merge}  we obtain a worse configuration.

\subsection{Redistribution of counters between levels $k-1$ and $k$ when $k\ge 2$ }\label{subsec:redistribute}
Consider again a general level $k$ such that either $2\le k<M$ 
or $k=M$ and $\alpha_M\ge 3$.

Here we obtain 
a new configuration 
by moving $\beta$ counters 
from level $k-1$ to $k$. 
If $\beta $ is negative, this means 
that we are moving $-\beta$ counters 
from level $k$ to level $k-1$. 
A natural constraint is  
\[
-\alpha_k\le\beta\le \alpha_{k-1}. 
\]

Here, we are interested in obtaining  a 
 worse configuration when the condition~\eqref{eq:merge}
 breaks down, that is 
\[
\alpha_{k-1}+\alpha_k
> N(B). 
\]
This means that the first factor in representation \eqref{Dk.beta} for $D_k(\beta)$ is negative. 
If $\alpha_{k}-\alpha_{k-1}$ is even,
then we can take 
\[
\beta=\frac{\alpha_{k-1}-\alpha_{k}}{2}
\] 
obtaining in the new configuration 
\[
\hat \alpha_k=\hat \alpha_{k-1}
=\frac{\alpha_k+\alpha_{k+1}}{2}.
\]
If the difference is odd, then we put 
$$
\beta=\frac{1}{2}+\frac{\alpha_{k-1}-\alpha_k}{2}.
$$
to obtain new configuration  
\[
\hat \alpha_{k}
=\alpha_k+\beta=
\frac{\alpha_k+\alpha_{k-1}+1}{2}
=\hat \alpha_{k-1}+1.
\]
Therefore, in this case we obtain a worse configuration by making the number of counters in levels $k-1$ and $k$ the same, or differing by one.

\section{Proof of Theorem~\ref{thm.upper.bound}}

Here, we pick $B=\frac{1}{2}$. Then, 
\[
N(B)=\frac{N+1}{2}. 
\]
This will make things easier to analyse, as we will have less cases to consider.

Consider first, $\alpha_M\in\{1,2\}$. 
If $\alpha_{M-1}+\alpha_{M-2}\le \frac{N+1}{2}$
we can obtain a worse configuration 
by merging level $M-1$ with $M-2$ 
using  Subsection~\ref{subsec:merge.not.1}. 
New configuration will have one level less and we keep repeating this procedure until 
\[
\alpha_{M-1}\le \frac{N+1}{2}, \quad 
\alpha_{M-1}+\alpha_{M-2}>
\frac{N+1}{2}.
\]
In the remaining  
levels we have 
\[
N-\alpha_M-\alpha_{M-1}-\alpha_{M-2}
< N-\frac{N+1}{2}\le \frac{N+1}{2}.
\]
Hence, we can merge the remaining levels into a single one. 

The same  reasoning applies when $\alpha_M\ge 3$. 
In this case we proceed until 
\[
\alpha_{M}+\alpha_{M-1}>
\frac{N+1}{2}
\]
and then merge the higher levels $\{1,2,\ldots, M-2\}$.

As a result 
either $M\le 3$ with $\alpha_M\ge 3$ or 
$M\le 4$ with $\alpha_M\in \{1,2\}$. 
After that we can rebalance the remaining $1,2$ or $3$ large levels to ensure 
that difference in a number of counters between them is at most $\frac{1}{2}$. 
Thus we have the following list of cases to analyse 
\begin{enumerate}[(i)]
\item $(N)$;
\item $(N-1,1)$;
\item $(N-2,2)$;
\item $\left(\left\lceil\frac{N-1}{2}\right\rceil,\left\lfloor\frac{N+1}{2}\right\rfloor\right)$;
\item $\left(\left\lceil\frac{N-2}{2}\right\rceil,\left\lfloor\frac{N}{2}\right\rfloor,1\right)$;
\item $\left(\left\lceil\frac{N-2}{2}\right\rceil,\left\lfloor\frac{N-2}{2}\right\rfloor,2\right)$;
\item $\left(\frac{N}{3},\frac{N}{3},\frac{N}{3}\right)$ if $N$ is divisible by $3$, 
$\left(\frac{N+2}{3},\frac{N-1}{3},\frac{N-1}{3}\right)$
if $N\mod 3 =1$, and 
$\left(\frac{N+1}{3},\frac{N+1}{3},\frac{N-2}{3}\right)$
if $N\mod 3 =2$;
\item $\left(\frac{N-1}{3},\frac{N-1}{3},\frac{N-1}{3},1\right)$ if $N\mod 3=1$, 
$\left(\frac{N+1}{3},\frac{N-2}{3},\frac{N-2}{3},1\right)$
if $N\mod 3 =2$, and 
$\left(\frac{N}{3},\frac{N}{3},\frac{N-3}{3},1\right)$
if $N\mod 3 =0$;
\item $\left(\frac{N-2}{3},\frac{N-2}{3},\frac{N-2}{3},2\right)$ if $N\mod 3=2$, 
$\left(\frac{N-1}{3},\frac{N-1}{3},\frac{N-4}{3},2\right)$
if $N\mod 3 =1$, and 
$\left(\frac{N}{3},\frac{N-3}{3},\frac{N-3}{3},2\right)$
if $N\mod 3 =0$;
\end{enumerate}

\subsection{Asymptotic analysis}

We can see that in the list of cases we have 
$k$ large levels for $k=1,2,3$, which might be followed by level of size $1$ or $2$. 

Thus, we need to analyse  
\begin{multline*}
\sum_{i=1}^k E^*\left(\frac{(i-1)N}{k}+\overline c_{i-1},\frac{N}{k}+c_k\right)
+E^*\left(N-1,1\right)
I(\alpha_M=1)
+E^*\left(N-2,2\right)
I(\alpha_M=2)
\\
-
\left(
\frac{
2f(N-1)I(\alpha_M=2)+
2(N-1)f(N-1)I(\alpha_M=1)}{N(N-1)}
\right),
\end{multline*}
where $\overline{c}_{i-1}$ and $c_k$ are some bounded constants. We have,
\begin{align*}
E^*(l,c) 
&=\frac{c(2l+c-1)\Delta f(l-1)
+
c(c-1)
\Delta f(l)
-c(c-1)
}{N(N-1)}\\
&=
\frac{c\left(
A(4l^2+(4c-6)l)
+B(2l+2c-2)+1-c
\right)
}{N(N-1)}
\end{align*}
We can see then that as $N\to\infty$, 
since $A\sim -\frac{1}{2N}$,
\[
E^*(N-1,1) \to 0,
\]
so this additional level does not make difference asymptotically. 
Also, 
\[
E^*\left(\frac{(i-1)N}{k}+\overline c_{i-1},\frac{N}{k}+c_k\right)\sim 
E^*\left(\frac{(i-1)N}{k},\frac{N}{k}\right)
.
\]
Since the contribution of terms with indicator 
is negative and the term with $I(\alpha_M=2)$ is negligible, we can conclude that asymptoically 
it is sufficient to analyse the following  three situations:
\begin{itemize}
\item $(N,1)$;
\item $(N/2,N/2,1)$;
\item $(N/3,N/3,N/3,1)$.
\end{itemize}
Note for $i\le k, k=1,2,3,$, as $N\to\infty$, 
since $A\sim -\frac{1}{2N}$, 
\begin{align*}
&E^*\left(\frac{(i-1)N}{k},\frac{N}{k}\right)\\
&\phantom{XX}\sim 
\frac{1}{k}
\frac{-\frac{1}{2N}\left(4\frac{(i-1)^2}{k^2}N^2
+4\frac{i-1}{k^2}N^2
\right)
+2B\frac{i}{k}N
-\frac{N}{k}
}{N}\\
&
\phantom{XX}\sim -\frac{2i(i-1)}{k^3}
+2B\frac{i}{k^2}
-\frac{1}{k^2}
.
\end{align*}
Hence, 
\begin{multline*}
\sum_{i=1}^k E^*\left(\frac{(i-1)N}{k},\frac{N}{k}\right)
-\frac{2(N-1)f(N-1)}{N(N-1)}\\
\sim 
\sum_{i=1}^k
\left(
-\frac{2i(i-1)}{k^3}
+2B\frac{i}{k^2}
-\frac{1}{k^2}
\right)
+1-2B\\
=
-2\frac{k^2-1}{3k^2}
+B\frac{k+1}{k}
-\frac{1}{k}
+1-2B=: S(k).
\end{multline*}
Recalling that we take $B=1/2$, we have
\[
S(k)=
-2\frac{k^2-1}{3k^2}
+\frac{k+1}{2k}-\frac{1}{k}.
\]
Then,  
\[
S(2)= -\frac{1}{4},\quad 
S(3)= -\frac{7}{27},\quad 
S(4)= -\frac{1}{4},
\]
hence the global minimum is $-\frac{7}{27}$.

\appendix

\section{Non-asymptotic upper bounds}
Let 
\[
S(\alpha) := \sum_{i=0}^{M-1} E^*(l_i,\alpha_{i+1}) -\frac{2}{N(N-1)} \Bigl( f(N-1)I(\alpha_M=2) +(N-1)f(N-1)I(\alpha_M=1) \Bigr) 
\]
We have computed 
 $S(\alpha)$ in all $9$ cases and subcases 
 using  Matlab. 
\begin{enumerate}[(i)]
\item $\alpha = (N)$: $S(\alpha) = 0$
\item $\alpha = (N-1, 1)$: $S(\alpha) = -\frac{N-1}{N\,\left(N-2\right)}$
\item $\alpha = (N-2, 2)$: $S(\alpha) = -\frac{2}{N}$
\item N even: $\alpha = (\frac{N}{2}, \frac{N}{2})$, $S(\alpha) = -\frac{N}{4\,\left(N-2\right)}$; N odd: $\alpha = (\frac{N}{2}-\frac{1}{2}, \frac{N}{2}+\frac{1}{2})$, $S(\alpha) = -\frac{N^2-1}{4\,N\,\left(N-2\right)}$
\item N even: $\alpha = (\frac{N}{2}, \frac{N}{2}-1, 1)$, $S(\alpha) = \frac{-N^3+N^2+2\,N-4}{4\,N\,\left(N^2-3\,N+2\right)}$; N odd: $\alpha = (\frac{N}{2}+\frac{1}{2}, \frac{N}{2}-\frac{1}{2}, 1)$, $S(\alpha) = -\frac{{\left(N+1\right)}^3}{4\,N\,\left(N^2-3\,N+2\right)}$
\item N even: $\alpha = (\frac{N}{2}-1, \frac{N}{2}-1, 2)$, $S(\alpha) = -\frac{N^2+N+2}{4\,N\,\left(N-1\right)}$; N odd: $\alpha = (\frac{N}{2}-\frac{1}{2}, \frac{N}{2}-\frac{3}{2}, 2)$, $S(\alpha) = -\frac{N^2-1}{4\,N\,\left(N-2\right)}$
\item N mod 3 = 0: $\alpha = (\frac{N}{3}, \frac{N}{3}, \frac{N}{3})$, $S(\alpha) = -\frac{N\,\left(7\,N-9\right)}{27\,\left(N^2-3\,N+2\right)}$; N mod 3 = 1: $\alpha = (\frac{N}{3}+\frac{2}{3}, \frac{N}{3}-\frac{1}{3}, \frac{N}{3}-\frac{1}{3})$, $S(\alpha) = \frac{-7\,N^2+2\,N+5}{27\,N\,\left(N-2\right)}$; N mod 3 = 2: $\alpha = (\frac{N}{3}+\frac{1}{3}, \frac{N}{3}+\frac{1}{3}, \frac{N}{3}-\frac{2}{3})$, $S(\alpha) = \frac{-7\,N^3+9\,N^2+3\,N-13}{27\,N\,\left(N^2-3\,N+2\right)}$
\item N mod 3 = 0: $\alpha = (\frac{N}{3}, \frac{N}{3}, \frac{N}{3}-1, 1)$, $S(\alpha) = -\frac{7\,N^3-12\,N^2+27}{27\,N\,\left(N^2-3\,N+2\right)}$; N mod 3 = 1: $\alpha = (\frac{N}{3}-\frac{1}{3}, \frac{N}{3}-\frac{1}{3}, \frac{N}{3}-\frac{1}{3}, 1)$, $S(\alpha) = \frac{-7\,N^2+5\,N+2}{27\,N\,\left(N-2\right)}$; N mod 3 = 2: $\alpha = (\frac{N}{3}+\frac{1}{3}, \frac{N}{3}-\frac{2}{3}, \frac{N}{3}-\frac{2}{3}, 1)$, $S(\alpha) = -\frac{7\,N^3-12\,N^2+19}{27\,N\,\left(N^2-3\,N+2\right)}$
\item N mod 3 = 0: $\alpha = (\frac{N}{3}, \frac{N}{3}-1, \frac{N}{3}-1, 2)$, $S(\alpha) = -\frac{7\,N^3-15\,N^2+27\,N-27}{27\,N\,\left(N^2-3\,N+2\right)}$; N mod 3 = 1: $\alpha = (\frac{N}{3}-\frac{1}{3}, \frac{N}{3}-\frac{1}{3}, \frac{N}{3}-\frac{4}{3}, 2)$, $S(\alpha) = -\frac{7\,N^2-8\,N+19}{27\,N\,\left(N-2\right)}$; N mod 3 = 2: $\alpha = (\frac{N}{3}-\frac{2}{3}, \frac{N}{3}-\frac{2}{3}, \frac{N}{3}-\frac{2}{3}, 2)$, $S(\alpha) = -\frac{7\,N^2-N+28}{27\,N\,\left(N-1\right)}$
\end{enumerate}



Comparing these values we have arrived to the upper bounds in the Remark~\ref{rem.finite.case}.

\section{Numerical results}

\subsection{Optimal Lyapunov function}
It is possible to obtain slightly better upper bounds for $V(N)$ by solving a linear optimisation problem in order 
to find  a Lyapunov 
function that is better than $f$ \
defined in~\eqref{defn.f}. 

Function $f$ is given in the table below. 

\begin{table}[h!]
\hspace*{-5cm}
\caption{Optimal function  $f(k)$ for each $N$.}
\resizebox{\textwidth}{!}{%
\begin{tabular}{|c|ccccccccccccccc|}
\hline
$N$ & $f(1)$ & $f(2)$ & $f(3)$ & $f(4)$ & $f(5)$ & $f(6)$ & $f(7)$ & $f(8)$ & $f(9)$ & $f(10)$ & $f(11)$ & $f(12)$ & $f(13)$ & $f(14)$ & $f(15)$  \\ \hline
4 & 0.57 & 0.43 & 1.43 &  &  &  &  &  &  &  &  &  &  &  &  \\
5 & 0.78 & 0.79 & 0.59 & 1.37 &  &  &  &  &  &  &  &  &  &  &  \\
6 & 0.77 & 0.88 & 0.82 & 0.57 & 1.35 &  &  &  &  &  &  &  &  &  &  \\
7 & 0.86 & 1.04 & 1.07 & 0.94 & 0.59 & 1.33 &  &  &  &  &  &  &  &  &  \\
8 & 0.95 & 1.19 & 1.29 & 1.26 & 1.03 & 0.63 & 1.32 &  &  &  &  &  &  &  &  \\
9 & 0.90 & 1.17 & 1.33 & 1.36 & 1.28 & 1.03 & 0.61 & 1.31 &  &  &  &  &  &  &  \\
10 & 0.96 & 1.26 & 1.45 & 1.55 & 1.53 & 1.37 & 1.07 & 0.62 & 1.30 &  &  &  &  &  &  \\
11 & 0.99 & 1.32 & 1.55 & 1.69 & 1.72 & 1.64 & 1.44 & 1.11 & 0.63 & 1.30 &  &  &  &  &  \\
12 & 0.97 & 1.31 & 1.56 & 1.74 & 1.82 & 1.82 & 1.69 & 1.46 & 1.11 & 0.62 & 1.29 &  &  &  &  \\
13 & 0.99 & 1.34 & 1.62 & 1.82 & 1.94 & 1.98 & 1.92 & 1.76 & 1.50 & 1.12 & 0.63 & 1.29 &  &  &  \\
14 & 1.02 & 1.39 & 1.69 & 1.92 & 2.07 & 2.15 & 2.13 & 2.03 & 1.84 & 1.55 & 1.15 & 0.63 & 1.28 &  &  \\
15 & 1.00 & 1.39 & 1.70 & 1.95 & 2.13 & 2.24 & 2.28 & 2.23 & 2.10 & 1.88 & 1.57 & 1.15 & 0.63 & 1.28 &  \\
16 & 1.01 & 1.40 & 1.73 & 1.99 & 2.19 & 2.32 & 2.39 & 2.38 & 2.30 & 2.14 & 1.89 & 1.57 & 1.15 & 0.63 & 1.28 \\
\hline
\end{tabular}
}
\end{table}
\begin{table}[h!]
\caption{Comparison of  theoretical and numerical bounds with simulated velocity.}
\resizebox{\textwidth}{!}{%
\begin{tabular}{|l|*{13}{c|}}
\hline
$N$ & 4 & 5 & 6 & 7 & 8 & 9 & 10 & 11 & 12 & 13 & 14 & 15 & 16 \\ \hline
Theoretical bound & 1.5 & 1.4 & 1.367 & 1.343 & 1.33 & 1.321 & 1.312 & 1.307 & 1.303 & 1.298 & 1.295 & 1.293 & 1.29 \\ \hline
Numerical bound & 1.43 & 1.37 & 1.35 & 1.33 & 1.32 & 1.31 & 1.30 & 1.30 & 1.29 & 1.29 & 1.28 & 1.28 & 1.28 \\ \hline
Simulated velocity & - & 1.35 & 1.33 & 1.31 & 1.3 & 1.29 & - & - & - & - & - & - & 1.27 \\ \hline
\end{tabular}
}
\end{table}


One can see for all $N$
that  parabola fits numerical data  very accurately. 
For example, when $N=16$ then 
\begin{align*}
f(k) &= -0.0405k^{2} + 0.5439k + 0.4574\\
&= -0.0405(k - 6.71)^{2} + 2.28
\end{align*}
results in the following 
graph. 
\begin{figure}[!h]
    \centering
\includegraphics[width=0.8\textwidth]{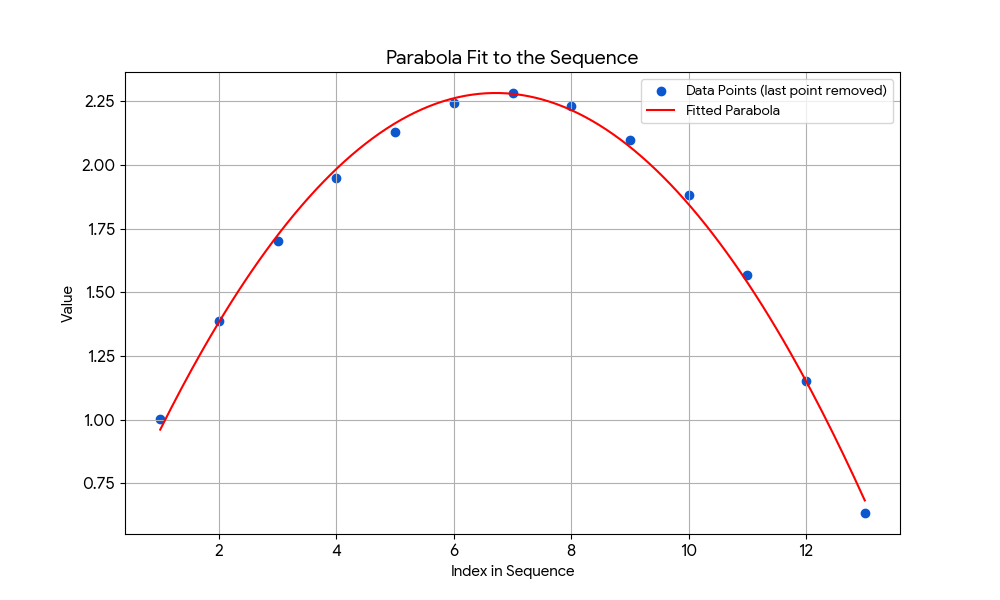}
    \caption{Fitting of $f$ by a parabola for $N=16$.}
    \label{15}
\end{figure}

\bibliographystyle{abbrv}
\bibliography{linear_programming}

\end{document}